\documentclass[11pt]{article}
\usepackage{a4wide}

\usepackage[a4paper, margin=2.3cm]{geometry}
\usepackage{amsmath,amssymb,amsthm, comment, enumitem, array}
\usepackage{color}
\usepackage{parskip}
\usepackage{hyperref}
\usepackage{parskip}
\usepackage{setspace}
\usepackage{graphicx}
\usepackage{mathtools}
\usepackage[thinc]{esdiff}
\usepackage[euler]{textgreek}

\usepackage{subfig}

\newtheorem{theorem}{Theorem}[section]
\newtheorem{proposition}[theorem]{Proposition}
\newtheorem{remark}{Remark}[section]
\newtheorem{corollary}[theorem]{Corollary}
\newtheorem{lemma}[theorem]{Lemma}
\newtheorem{question}[theorem]{Question}

\newtheorem*{definition*}{Definition}

\hypersetup{
    colorlinks=true,
    linkcolor = blue,
    citecolor=magenta,
    urlcolor=cyan,
}

\usepackage[capitalise, nameinlink]{cleveref}
\crefname{equation}{}{}
\crefname{figure}{{\sc Figure}}{{\sc Figure}}
\crefname{subsection}{Subsection}{Subsections}

\begin{document}
\title{Triangles with one fixed side--length, a Furstenberg type problem, and incidences in finite vector spaces}
\author{Thang Pham\thanks{University of Science, Vietnam National University, Hanoi. Email: phamanhthang.vnu@gmail.com}}
\date{}
\maketitle
\begin{abstract}
The first goal of this paper is to prove a sharp condition to guarantee of having a positive proportion of all congruence classes of triangles in given sets in $\mathbb{F}_q^2$. More precisely, for $A, B, C\subset \mathbb{F}_q^2$, if $|A||B||C|^{1/2}\gg q^4$, then for any $\lambda\in \mathbb{F}_q\setminus \{0\}$, the number of congruence classes of triangles with vertices in $A\times B\times C$ and one side-length $\lambda$ is at least $\gg q^2$. In higher dimensions, we obtain similar results for $k$-simplex but under a slightly stronger condition. Compared to the well--known $L^2$ method in the literature, our approach offers better results in both conditions and conclusions. When $A=B=C$, the second goal of this paper is to give a new and unified proof of the best current results on the distribution of simplex due to Bennett, Hart, Iosevich, Pakianathan, and Rudnev (2017) and McDonald (2020). The third goal of this paper is to study a Furstenberg type problem associated to a set of rigid motions. The main ingredients in our proofs are incidence bounds between points and rigid motions. While the incidence bounds for large sets are due to the author and Semin Yoo (2023), the bound for small sets will be proved by using a point--line incidence bound in $\mathbb{F}_q^3$ due to Koll\'{a}r (2015). 
\end{abstract}

\textbf{Mathematics Subject Classification}: 52C10; 51E99

\tableofcontents
\section{Introduction}
Let $\mathbb{F}_q$ be a finite field of order $q$, where $q$ is a prime power. 

Let $\mathbf{x}=(x^1, \ldots, x^{k+1})$ and $\mathbf{y}=(y^1, \ldots, y^{k+1})$ be two $k$-simplices in $\mathbb{F}_q^d$. We say that these two simplices are in the same congruence class, denoted by $(x^1, \ldots, x^{k+1})\sim (y^1, \ldots, y^{k+1})$, if there exist $g\in O(d)$ and $z\in \mathbb{F}_q^d$ such that $gx^i+z=y^i$ for all $1\le i\le k+1$.  For each $k$-simplex $\mathbf{x}=(x^1, \ldots, x^{k+1})$, by $\dim(\mathbf{x})$ we mean the dimension of the space spanned by vectors $\{x^2-x^1, \ldots, x^{k+1}-x^1\}$. The $k$-simplex $\mathbf{x}$ is called non-degenerate if $\dim(\mathbf{x})=\min\{k, d\}$ and degenerate otherwise.

Let $A_1, \ldots, A_{k+1}$ be subsets of $\mathbb{F}_q^d$, we say the $k$-simplex $\mathbf{x}=(x^1, \ldots, x^{k+1})$ has vertices in $\prod_{i=1}^{k+1}A_i$ if $x^i\in A_i$ for all $1\le i\le k+1$. In this paper, we study the following question. 

\begin{question}\label{ques1}
Let $2\le k\le d$ be an integer, and $A_1, \ldots, A_{k+1}$ be subsets of $\mathbb{F}_q^d$. Given a $(k-1)$-non-degenerate simplex $\Delta_{k-1}$,
what conditions on $A_i$ do we need to ensure that the number of congruence classes of $k$-simplices with vertices in $\prod_{i=1}^{k+1}A_i$ containing a copy of $\Delta_{k-1}$ is at least $\gg q^k$?
\end{question}
Our first result is on the case $d=k=2$.
\begin{theorem}\label{firstmaintheorem}
Let $A, B, C$ be sets in $\mathbb{F}_q^2$ with $q\equiv 3\mod 4$. Assume that 
\[|A||B||C|^{1/2}\gg q^4\]
then, for any $\lambda\in \mathbb{F}_q\setminus \{0\}$, the number of congruence classes of triangles with vertices in $A\times B\times C$ and one edge of length $\lambda$ is at least $\gg q^2$.
As a consequence, if $|A||B||C|^{1/2}\gg q^4$, then the number of congruence classes of triangles with vertices in $A\times B\times C$ is at least $\gg q^3$.
\end{theorem}
This theorem is sharp in the sense that the lower bound $q^4$ can not be improved to $q^{4-\epsilon}$ for any $\epsilon>0$. In particular, for any $\epsilon>0$, there exist $A, B, C\subset \mathbb{F}_q^2$ with $|A||B||C|^{1/2}\sim q^{4-2\epsilon}$ such that the number of congruence classes of triangles with vertices in $A\times B\times C$ is at most $q^{3-\epsilon}$. A detailed construction will be provided in Section \ref{sec:sh}. 

The study of this type of question was introduced by Furstenberg, Katznelson and Weiss \cite{FW}, namely, they proved that for $A\subset \mathbb{R}^2$, if $A$ has positive upper Lebesgue density, then the $\delta$-neighborhood of $A$, for any $\delta>0$, contains a congruence copy of a large dilate of every three points configuration. We note that taking the $\delta$-neighborhood of $A$ is necessary, this comes from an example of Bourgain \cite{B} with the three points forming an arithmetic progression. Bourgain also extended this result for any $k$-simplex with $k<d$. In the integer lattice $\mathbb{Z}^d$, Magyar \cite{Ma1, Ma2} studied similar questions, in particular, he showed that if $A\subset \mathbb{Z}^d$, $d>2k+4$, has positive upper density, then $A$ contains all large dilates of a given $k$-simplex. We refer the interested reader to \cite{Ma1, Ma2, Ma3} and references therein for more discussions and recent results in this direction. 

In the setting of finite fields, the first result on Question \ref{ques1} was given by Bennett, Iosevich, and Pakianathan \cite{alex} in 2014. They proved the following theorem. 
\begin{theorem}[\cite{alex}]\label{BIP}
    Let $A$ be a set in $\mathbb{F}_q^2$ with $q\equiv 3\mod 4$. Assume that 
    $|A|\gg q^{7/4}$, then for any $\lambda\in \mathbb{F}_q\setminus \{0\}$, the number of congruence classes of triangles with vertices in $A\times A\times A$ and one edge of length $\lambda$ is at least $\gg q^2$.
\end{theorem}
To compare with Theorem \ref{firstmaintheorem}, assume that $A=B=C$, then we will need the condition that $|A|\gg q^{8/5}$. This improves the exponent $7/4$ from the previous result. We now sketch the main idea for this improvement. For a rigid motion $r=(g,z)\in O(2)\times \mathbb{F}_q^2$, we say that the point $(u, v)\in \mathbb{F}_q^2\times \mathbb{F}_q^2$ is incident to $r$ if $\phi_r(u):=gu+z=v$. When $|A|$ is large enough, say $|A|\gg q^{3/2}$, then the number of pairs $(u, v)\in A\times A$ such that $||u-v||=t$, $t\ne 0$, is about $|A|^2/q$. Fix one of those pairs, say $(u_0, v_0)$, for each pair $(u, v)$ with $||u-v||=t$, there exists unique a rigid motion $r=(g, z)$ that maps $u$ to $u_0$ and $v$ to $v_0$. We denote the set of corresponding rigid motions by $R$. The proof is now reduced to a Furstenberg type problem of showing that the set
\begin{equation}\label{set}\bigcup_{r\in R}\phi_r(A)\end{equation}
has size of at least $\gg q^2$. To bounding the size of this set, we will need a point-rigid motion incidence bound. Compared to the proof in \cite{alex}, there are two new perspectives in our proof: a stronger incidence theorem (Theorem \ref{thm-inci-large-2}) due to the author and Yoo \cite{PY}, and a more effective mechanism to apply this theorem. 

We want to make a remark here that in both Theorems \ref{firstmaintheorem} and \ref{BIP}, the condition $q\equiv 3\mod 4$ is required. This is needed in the proofs of incidence bounds. It would be interesting to see if the same result holds when $q\equiv 1\mod 4$. 

The same approach can be applied for the problem of $k$-simplex in higher dimensions, but with a stronger condition.
\begin{theorem}\label{moregeneral}
Given a $(k-1)$-non-degenerate simplex $\Delta_{k-1}$ with nonzero side-lengths in $\mathbb{F}_q^d$. Assume $\prod_{i=1}^{k+1}|A_i|\gg q^{dk+1}$ and $|A_i|\gg q^{\frac{d-1}{2}+k-1}$ for all $1\le i\le k$, then the number of congruence classes of $k$-simplices containing a copy of $\Delta_{k-1}$ with vertices in $\prod_{i=1}^{k+1}A_i$ is at least $\gg q^k$.
\end{theorem}
When $d=2$ and $k=2$, the above theorem tells us that the three sets $A, B, C$ have to satisfy $|A||B||C|\gg q^5$, which is worse than the condition $|A||B||C|^{1/2}\gg q^{4}$ of Theorem \ref{firstmaintheorem}. This is because the proof of Theorem \ref{moregeneral} involves a more general incidence theorem which is not very strong in some specific dimensions. One more remark we want to add here is that the condition $q\equiv 3\mod 4$ is not required in this theorem.

If we only count the number of congruence classes of $k$-simplex in a given set, the next theorem due to Bennett, Hart, Iosevich, Pakianathan, and Rudnev \cite{iosevich-Forum}, and McDonald \cite{Mc-d}.
\begin{theorem}[\cite{iosevich-Forum, Mc-d}]\label{thm:positive}
Let $A$ be a set in $\mathbb{F}_q^d$. 
\begin{enumerate}
    \item If $1\le k\le d$ and $|A|\gg q^{\frac{dk+1}{k+1}}$, then the number of congruence classes of non-degenerate $k$-simplex in $A$ is at least $\gg q^{\binom{k+1}{2}}$. 
    \item If $k> d$ and $|A|\gg q^{\frac{dk+1}{k+1}}$, then the number of congruence classes of non-degenerate $k$-simplex in $A$ is at least $q^{d(k+1)-\binom{d+1}{2}}$.
\end{enumerate}
\end{theorem}
The main idea in the proof of this theorem is to bound the $L^2$-norm of $k$-simplices. In this paper, we present a unified proof using results on the number of $k$-rich rigid motions. This approach also leads to the following improvement in two dimensions.
\begin{theorem}\label{improve-2d}
For $A\subset \mathbb{F}_q^2$ with $|A|\gg q^{\frac{4k}{2k+1}}$ and $q\equiv 3\mod 4$, we have the number of congruence classes of non-degenerate $k$-simplex in $A$ is at least $q^{2k-1}$.
\end{theorem}
Unlike Theorem \ref{firstmaintheorem}, we do not have any constructions on the sharpness of Theorems \ref{thm:positive} and \ref{improve-2d}. For $k\le d$, let $\alpha_{k,d}$ be the smallest exponent such that the number of congruence classes of $k$-simplex in $A$ is at least $\gg q^{\binom{k+1}{2}}$ whenever $|A|\gg q^{\alpha_{k, d}}$. The best current lower bounds of $\alpha_{k, d}$ are recorded in the following theorem.
\begin{theorem}[\cite{iosevich-Forum}] Let $k$ and $d$ be integers with  $k\le d$.

\vskip.125in
\begin{itemize}
\item[i.]If $k=1$ and $d\ge 3$ odd, we have $\alpha_{1,d} \ge \frac{d+1}{2}$. 

\item[ii.] If $k=1$ and $d \ge 2$ is even, we have $\alpha_{1,d} \ge \frac{d}{2}$. 

\item[iii.] If $k>1$, we have $\alpha_{k,d} \ge k-1+\frac{1}{k}$.
\end{itemize}
\end{theorem}

Assume $k\le d$, if $\prod_{i=1}^{k+1}|A_i|\gg q^{dk+1}$, then, by Theorem \ref{moregeneral}, the number of congruence classes of $k$-simplices in with vertices in $\prod_{i=1}^{k+1}A_i$ is at least $\gg q^{\binom{k+1}{2}}$. One might ask if the method in the proof of Theorem \ref{thm:positive} also works for different sets. The answer is positive, but stronger assumptions are needed. For instance, when $k=2$, we will need $|A_1|, |A_2|, |A_3|\gg q^{\frac{2d+1}{3}}$ or $|A_1|, |A_2|\gg q^{\frac{3d+1}{4}}$ and $|A_3|\gg q^{\frac{d+1}{2}}$. More details will be provided in Section \ref{sec:5}.

The last theorems of this paper are on a Furstenberg type problem motivated by bounding the set in (\ref{set}). We study the following question. 
\begin{question}
    Let $R$ be a set of rigid motions in $\mathbb{F}_q^d$ and $A\subset \mathbb{F}_q^d$, what can we say about the lower bound of the set $\bigcup_{r\in R}\phi_r(A)$?
\end{question}
Since the size of $\bigcup_{r\in R}\phi_r(A)$ is always bounded from below by $|A|$, we are interested in improving this trivial lower bound. 

Our first result reads as follows. 
\begin{theorem}\label{Fur1}
Let $A$ be a set of points and $R$ be a set of rigid motions in $\mathbb{F}_q^d$. 
We have 
\[\left\vert \bigcup_{r\in R}\phi_r (A)\right\vert\gg \min \left\lbrace q^d, ~\frac{|A||R|}{q^d|O(d-1)|}\right\rbrace.\]
\end{theorem}
In order to have the RHS of the above inequality larger than $|A|$, we need the condition that $|R|\gg q^d|O(d-1)|$. We emphasize here that this condition does not depend on the size of $A$, and is sharp if no additional assumption on $A$ is added. To see this, let $X\subset \mathbb{F}_q$, we view $O(d-1)$ as a subset of $O(d)$ and define $Z:=\{(0, \ldots, 0, *)\colon *\in X\}$. Set $A=(\mathbb{F}_q^{d-1}\times \{0\})+Z$ and $R=O(d-1)\times A$. It is clear that $|R|=|O(d-1)||X|q^{d-1}$ and 
\[\bigcup_{r\in R}\phi_r (A)\subset \mathbb{F}_q^{d-1}\times (X+X).\]
Thus, given $0<\epsilon<1$, if we take $X$ to be an arithmetic progression in $\mathbb{F}_q$ of length $q^{1-\epsilon}$, then $|R|=|O(d-1)|q^{d-\epsilon}$ and 
\[|\bigcup_{r\in R}\phi_r (A)|\le 2|A|.\]
In the next theorems, we provide improvements in which either the size of $A$ is small or there is a relation between the sizes of $A$ and $R$. 
\begin{theorem}\label{Fur2}
Let $A$ be a set of points and $R$ be a set of rigid motions in $\mathbb{F}_q^d$. Assume in addition that either ($d\ge 3$ odd) or ($d\equiv 2\mod 4$ and $q\equiv 3\mod 4$). 
\begin{enumerate}
    \item[\textup{(1)}] If $|A|<q^{\frac{d-1}{2}}$, then we have 
\[\left\vert \bigcup_{r\in R}\phi_r (A)\right\vert\gg \min \left\lbrace q^d, ~\frac{|A||R|}{q^{d-1}|O(d-1)|}\right\rbrace.\]
    \item[\textup{(2)}] If $q^{\frac{d-1}{2}}\le |A|\le q^{\frac{d+1}{2}}$, then we have
\[\left\vert \bigcup_{r\in R}\phi_r (A)\right\vert\gg \min \left\lbrace q^d, ~\frac{|R|}{q^\frac{d-1}{2}|O(d-1)|}\right\rbrace.\]
\end{enumerate}
\end{theorem}
\begin{theorem}\label{Fur3}
    Let $A$ be a set of points and $R$ be a set of rigid motions in $\mathbb{F}_q^2$ with $q\equiv 3\mod 4$. We have 
    \[\left\vert \bigcup_{r\in R}\phi_r (A)\right\vert\gg \min \left\lbrace q^2, ~\frac{|A|^{1/2}|R|}{q}\right\rbrace.\]
\end{theorem}
When $A$ and $R$ are small sets and $R$ is a subset of oriented rigid motions, it is possible to obtain a further improvement. Let $SO(2, q)$ be the group of orthogonal matrices with determinant $1$. Let $T(2, q)$ be the group of translations in $\mathbb{F}_q^2$, and $SF(2, q)$ be the group of positively oriented rigid motions in $\mathbb{F}_q^2$, i.e. the group of elements of the from $t\circ g$, where $t\in T(2, q)$ and $g\in SO(2, q)$. Define $SF'(2, q)=SF(2, q)\setminus T(2, q)$.
\begin{theorem}\label{Fur4}
    Let $A$ be a set of points in $\mathbb{F}_q^2$ and $R$ be a set of rigid motions in $SF'(2, q)$ with $q\equiv 3\mod 4$. Assume that $2|R|^{1/5}<|A|<|R|^{3/5}$, then 
   we have 
    \[\left\vert \bigcup_{r\in R}\phi_r (A)\right\vert\gg |R|^{3/5}.\]
\end{theorem}
While we use incidence theorems developed in \cite{PY} to prove Theorems \ref{Fur1}, \ref{Fur2}, and \ref{Fur3}, the proof of Theorem \ref{Fur4} requires a new incidence theorem for small sets of points and rigid motions (Theorem \ref{thm-inci}). This incidence theorem will be proved by using a point-line incidence bound in $\mathbb{F}_q^3$ due to Koll\'{a}r \cite{kollar}. We do not have any answer to the question on the sharpness of Theorems \ref{Fur2}, \ref{Fur3}, and \ref{Fur4}, so we leave it as an open problem. We conclude this paper with some discussions on the case of prime fields. 
\section{Incidences between points and rigid motions}
In this section, we recall incidence theorems between points and rigid motions in $\mathbb{F}_q^d$ from \cite{PY} and state a new incidence theorem for small sets in $\mathbb{F}_q^2$. Let $P=U\times V\subset \mathbb{F}_q^d\times \mathbb{F}_q^d$ and $R$ be a subset of $O(d)\times \mathbb{F}_q^d$. We say the point $(u, v)$ is incident to the motion $r=(g, z)$ if $\phi_r(u):=gu+z$ is equal to $v$. We denote the number of incidences between $P$ and $R$ by $I(P, R)$. 
\begin{theorem}[\cite{PY}]\label{inci-large}
    Let $P\subset \mathbb{F}_q^{2d}$ and $R$ be a subset of rigid motions in $\mathbb{F}_q^d$. Then the number of incidences between $P$ and $R$ satisfies
    \[I(P, R)\le \frac{|P||R|}{q^d}+Cq^{d/2}\sqrt{|O(d-1)|}\sqrt{|P||R|},\]
    for some large positive constant $C$.
\end{theorem}
\begin{corollary}\label{co-rich-d}
Let $P=U\times U\subset \mathbb{F}_q^d\times \mathbb{F}_q^d$. For $k>2|P|q^{-d}$, let $R_k$ be the number of elements of $R$ incident to at least $k$ elements from $P$. Then we have 
\[|R_k|\ll \frac{|O(d-1)||P|q^d}{k^2}.\]
\end{corollary}
\begin{proof}
We have
    \[k|R_k|\le I(P, R_k)\le \frac{|P||R_k|}{q^d}+Cq^{d/2}\sqrt{|O(d-1)|}\sqrt{|P||R_k|}.\]
Solving this inequality gives us the desired bound.
\end{proof}
If we put more conditions on $d$ and $q$, the next two theorems present improvements. 
\begin{theorem}[\cite{PY}]\label{thm: incidence1}
     Let $P=U\times V$ for $U, V\subset \mathbb{F}_q^d$. Assume in addition that either ($d\ge 3$ odd) or ($d\equiv 2\mod 4$ and $q\equiv 3\mod 4$). There exists a large positive constant $C$ such that the following hold.
\begin{enumerate}
    \item[\textup{(1)}] If $|U|<q^{\frac{d-1}{2}}$, then we have 
\[ I(P, R)\le \frac{|P||R|}{q^d}+Cq^{\frac{d-1}{2}}\sqrt{|O(d-1)|}|P|^{1/2}|R|^{1/2}.\]
    \item[\textup{(2)}] If $q^{\frac{d-1}{2}}\le |U|\le q^{\frac{d+1}{2}}$, then we have
\[I(P, R)\le \frac{|P||R|}{q^d}+Cq^{\frac{d-1}{4}}\sqrt{|O(d-1)|}|P|^{1/2}|R|^{1/2}|U|^{1/2}.\]
\end{enumerate}
\end{theorem}
\begin{theorem}[\cite{PY}]\label{thm-inci-large-2}
Let $P=U\times V\subset \mathbb{F}_q^2\times \mathbb{F}_q^2$, $|U|\le |V|$, and $R$ be a subset of $O(2)\times \mathbb{F}_q^2$ with $q\equiv 3\mod 4$. Then we have
\[I(P, R)\le \frac{|P||R|}{q^2}+Cq^{1/2}|U|^{3/4}|V|^{1/2}|R|^{1/2},\]
for some large positive constant $C$.
\end{theorem}
\begin{corollary}\label{co-rich-2}
Let $P=U\times U\subset \mathbb{F}_q^2\times \mathbb{F}_q^2$ with $q\equiv 3\mod 4$. For $k>2|P|q^{-2}$, let $R_k$ be the number of elements of $R$ incident to at least $k$ elements from $P$. Then we have 
\[|R_k|\ll \frac{q|P|^{5/4}}{k^2}.\]
\end{corollary}
\begin{proof}
The proof is the same as that of Corollary \ref{co-rich-d}, namely, 
    \[k|R_k|\le I(P, R_k)\le \frac{|P||R_k|}{q^2}+Cq^{1/2}|U|^{5/4}|R_k|^{1/2}.\]
Solving this inequality gives us the desired bound.
\end{proof}


On the sharpness, it has been mentioned in \cite{PY} that Theorems \ref{inci-large} and \ref{thm: incidence1} are sharp in odd dimensions. In terms of applications (Theorem \ref{firstmaintheorem}), Theorem \ref{thm-inci-large-2} is sharp in sense that the term $q^{\frac{1}{2}}|U|^{3/4}|V|^{1/2}|R|^{\frac{1}{2}}$ cannot be improved to $q^{\frac{1}{2}-\epsilon}|U|^{3/4}|V|^{1/2}|R|^{\frac{1}{2}}$ for any $\epsilon>0$. 

As proved in \cite[Section 3]{PY} that one can use the Cauchy-Schwarz inequality to have 
\begin{equation}\label{diss:1}I(P, R)\le |P||R|^{1/2}+|R|.\end{equation}
Assume $|P|=|R|=N$, so these two incidence bounds offer the upper bound of $N^{3/2}$. Note that when $N$ is small enough, this is better than that of Theorem \ref{thm-inci-large-2} which depends on $q$. When $q$ is a prime number and $P=U\times U$ with $|U|\ll q$, it has been proved in \cite[Section 3]{PY} that 
\begin{equation}\label{diss:3}I(P, R)\ll |P|^{5/6}|R|^{1/2}+|R|.\end{equation}
In particular, if $|P|=|R|=N$, then $I(P, R)\ll N^{4/3}$. 

If we put (\ref{diss:1}) and (\ref{diss:3}) together, the following questions appear naturally:

\begin{question}
Suppose $|P|=|R|=N$, could we get an upper bound of the form $N^{\frac{3}{2}-\epsilon}$, for some $\epsilon>0$, over arbitrary finite fields?
\end{question}

\begin{question}
    Is it possible to have an upper bound of the form 
    \[I(P, R)\ll |P||R|^{\frac{1}{2}-\epsilon}\]
    for some $\epsilon>0$?
\end{question}
In the next theorem, we address these two questions for oriented rigid motions.
\begin{theorem}\label{thm-inci}
    Let $P=U\times U\subset \mathbb{F}_q^4$ and $R$ be a subset of $SF'(2, q)$ with $q\equiv 3\mod 4$. The number of incidences between $P$ and $R$, denoted by $I(P, R)$, satisfies
    \[I(P, R)\ll |P||R|^{2/5}+|R|^{6/5}.\]
In particular, if $|P|=|R|=N$, then we have 
\[I(P, R)\ll N^{\frac{3}{2}-\frac{1}{10}}.\]
\end{theorem}
We remark here that our argument in the proof of Theorem \ref{thm-inci} also works for other sets $P$ satisfying the condition that $\min\{|\pi_{12}(P|), |\pi_{34}(P)|\}\ll |P|^{1/2}$, where $\pi_{ij}(P)$ is the projection of $P$ onto the two coordinates $i$ and $j$. 

In the above theorem, the term $|P||R|^{2/5}$ cannot be decreased to lower than $|P||R|^{1/3}$. The reason is that one can take $q=p^3$ with $p\equiv 3\mod 4$, $P=U\times U\subset\mathbb{F}_p^2\times \mathbb{F}_p^2$ with $|P|=p^3$, and $R=O(2, p)\times \mathbb{F}_p^2$. Here $O(2, p)$ is the set of orthogonal matrices with entries in $\mathbb{F}_p$. Note that $q=p^r$ is $3\mod 4$ if and only if $p\equiv 3\mod 4$ and $r$ is odd. Thus, $q=p^3$ with $p\equiv 3\mod 4$ satisfies the property $q\equiv 3\mod 4$. We can see that each point in $P$ is incident to about $|O(2, p)|=p$ elements of $R$. So, $I(P, R)\sim |P||R|^{1/3}$.  

\section{Triangles with one fixed side-length (Theorem \ref{firstmaintheorem})}\label{sec:sh}
To prove Theorem \ref{firstmaintheorem}, we recall the following result due to Shparlinski \cite{Sh} on the number of ``unit" distances in a pair of given sets in $\mathbb{F}_q^d$. When the two sets are the same, this was proved by Iosevich and Rudnev in \cite{IR}.
\begin{theorem}\label{unit-distance}
    Let $A, B\subset \mathbb{F}_q^2$. For $\lambda\in \mathbb{F}_q^*$, let $N(\lambda)$ be the number of pairs $(x, y)\in A\times B$ such that $||x-y||=\lambda$. Then we have 
    \[\frac{|A||B|}{q}-4q^{\frac{1}{2}}\sqrt{|A||B|}\le N(\lambda)\le \frac{|A||B|}{q}+4q^{\frac{1}{2}}\sqrt{|A||B|}.\]
\end{theorem}

Let $(x, y)$ be a line segment of length  $||x-y||=\lambda_1\ne 0$ in $\mathbb{F}_q^2$. For $\lambda_2, \lambda_3\in \mathbb{F}_q$, it has been proved in \cite{alex} that this segment can be extended into at most two triangles $(x, y, z)$ with $||x-z||=\lambda_2$ and $||y-z||=\lambda_3$. The precise statement is as follows. 
\begin{lemma}\label{fact}
    Let $(x, y) $be a line segment of length  $||x-y||=\lambda_1\ne 0$ in $\mathbb{F}_q^2$. For $\lambda_2, \lambda_3\in \mathbb{F}_q$, this segment can be extended into exactly $\mu$ triangles $(x, y, z)$ with $||x-z||=\lambda_2$ and $||y-z||=\lambda_3$, where 
    \[\mu=\begin{cases} &2~~\mbox{if $4\sigma_2-\sigma_1^2$ is a non-zero square in $\mathbb{F}_q$}\\
    &1~~\mbox{if $4\sigma_2-\sigma_1^2$ is zero}\\
    &0~~\mbox{if $4\sigma_2-\sigma_1^2$ is a non-square in $\mathbb{F}_q$}\end{cases}\]
and $\sigma_1=\lambda_1+\lambda_2+\lambda_3$, $\sigma_2=\lambda_1\lambda_2+\lambda_2\lambda_3+\lambda_3\lambda_1$.
\end{lemma}
With these results, we are ready to prove Theorem \ref{firstmaintheorem}.
\begin{proof}[Proof of Theorem \ref{firstmaintheorem}]
Given $\lambda\ne 0$, we know from Theorem \ref{unit-distance} that the number of pairs $(x, y)\in A\times B$ such that $||x-y||=\lambda$ is about $\sim \frac{|A||B|}{q}$. Let $(x_0, y_0)$ be one of those pairs, for any pair $(x, y)\in A\times B$ such that $(x, y)\ne (x_0, y_0)$ and $||x-y||=\lambda$, there exists unique $r=(g, z)\in O(2)\times \mathbb{F}_q^2$ such that $gx_0+z=x$ and $gy_0+z=y$. Let $R$ be the set of all such $(g, z)$ corresponding to all pairs $(x, y)$ with $||x-y||=\lambda$ in $A\times B$. So, $|R|\sim \frac{|A||B|}{q}$. 

For $r=(g, z)\in O(2)\times \mathbb{F}_q^2$, we define the map $\phi_r$ (or $\phi_{g, z}$ to be precise) from $\mathbb{F}_q^2$ to $\mathbb{F}_q^2$ by
\[\phi_{r}(x)=gx+z.\]

Set $C':=\bigcup_{r\in R}\phi_r^{-1}(C)$. It is enough to prove that
\begin{equation}\label{eq:11}|C'|\gg \min \left\lbrace q^2, ~\frac{|A||B||C|^{1/2}}{q^{2}} \right\rbrace\gg q^2,\end{equation}
whenever $|A||B||C|^{1/2}\gg q^{4}$. 

To see why we have at least $\gg q^2$ distinct congruence classes of triangles with one side-length $\lambda$, we observe that for each $v\in C'$, there exists $(x, y, u)\in A\times B\times C$ and $(g, z)\in R$ such that $\phi_{g, z}^{-1}(u)=v$, $\phi_{g, z}(x_0)=x$, and $\phi_{g, z}(y_0)=y$. This means that the two triangles with vertices $(x_0, y_0, v)$ and $(x, y, u)$ are in the same congruence class. Using Lemma \ref{fact}, we conclude that the number of distinct congruence classes of triangles is at least $\gg q^2$.

To prove (\ref{eq:11}), we observe that $I(C\times C', R^{-1})=|R||C|$. Applying Theorem \ref{thm-inci-large-2} gives us 
\[|R||C|\ll \frac{|C||C'||R|}{q^2}+q^{1/2}|C|^{3/4}|C'|^{1/2}|R|^{1/2}.\]
Solving this inequality infers (\ref{eq:11}). 
This completes the proof of the theorem.
\end{proof}
{\textbf{Sharpness:}} To see the sharpness of this theorem, we construct an example as follows. Let $C=\{(0, 0)\}\subset \mathbb{F}_q^2$ and $0<\epsilon<1$.  
We know that the group of rotations $SO(2, q)$ in $\mathbb{F}_q^2$ with $q\equiv 3\mod 4$ is cyclic of order $q+1$. Assume $q=p^r$, $r$ is an exponent of three, and large enough such that $(q+1)/(p+1)\sim q^{1-\epsilon}$. We recall the fact that $q\equiv 3\mod 4$ if and only if $p\equiv 3\mod 4$ and $r$ is odd. Let $\theta\in SO(2, q)$ be an element of order $k \sim q^{1-\epsilon}$. Let $X$ be the set of $t\in \mathbb{F}_q\setminus \{0, 1\}$ such that if $t\in X$ then $-t\not \in X$. Let $v\in \mathbb{F}_q^2$ with $||v||=1$, define
\[A=\{v, \theta v, \cdots, \theta^{k-1} v\}\bigcup_{t\in X} \{tv, t\theta v, \cdots, t\theta^{k-1} v\}.\]
We have $|A|=(|X|+1)\cdot k\sim q^{2-\epsilon}$. We write 
\[A_1=\{v, \theta v, \cdots, \theta^{k-1} v\}, ~A_t=\{tv, t\theta v, \cdots, t\theta^{k-1} v\}, ~t\in X.\]
By a direct computation, one can check that the set of distances between $A_\lambda$ and $A_\beta$, $\lambda, \beta\in X\cup\{1\}$, is at most $q^{1-\epsilon}$. This implies the number of congruence classes of triangles with one vertex in $C$ and the other two vertices in $A$ is at most $q^{3-\epsilon}$.

In other words, we have proved that for $0<\epsilon<1$, there exists $q=q(\epsilon)$ large enough such that there exist $A, B, C$ with $|A||B||C|^{1/2}=q^{4-2\epsilon}$ and the number of congruence classes of triangles with vertices in $A\times B\times C$ is at most $q^{3-\epsilon}$.
\section{Extension in higher dimensions for simplex (Theorem \ref{moregeneral})}
The idea to prove Theorem \ref{moregeneral} is the same as that of Theorem \ref{firstmaintheorem}. We need a more general version of Theorem \ref{unit-distance} for $(k-1)$-simplex due to Vinh \cite{vinh1, vinh2}. 
\begin{theorem}\label{v-simplex}
    Let $A_1, \ldots, A_{k}\subset \mathbb{F}_q^d$ with $|A_i|\gg q^{\frac{d-1}{2}+k-1}$. Then for any given $(k-1)$-simplex with non-zero side-lengths $\Delta_{k-1}$, the number of copies of $\Delta_{k-1}$ in $\prod_{i=1}^{k}A_i$ is $\sim q^{-\binom{k}{2}}\cdot \prod_{i=1}^k|A_i|$.
\end{theorem}
The following extends Lemma \ref{fact} to higher dimensions, a proof can be found in \cite[Lemma 3]{ABC}. 
\begin{lemma}\label{fact2}
Suppose we have $k$ spheres $S_1, \ldots, S_k$ with centers $a^1, \ldots, a^k$ and non-zero radii such that the system $\{a^2-a^1, \ldots, a^k-a^1\}$ is linear independent, then $|S_1\cap S_2\cap \cdots\cap S_k|\le 2q^{d-k}$. 
\end{lemma}
We are ready to provide proof of Theorem \ref{moregeneral}. Its proof is the same as that of Theorem \ref{firstmaintheorem}, except that Theorem \ref{thm-inci-large-2} is replaced by Theorem \ref{inci-large}.
\begin{proof}[Proof of Theorem \ref{moregeneral}]
    Using Theorem \ref{v-simplex}, we know that there are about $\sim \prod_{i=1}^k|A_i|q^{-\binom{k}{2}}$ copies of $\Delta_{k-1}$ in $A$. We fix one of them, say $\Delta_{k-1}^0$. Then for each copy $\Delta_{k-1}^i$, there are at least $|O(d-k+1)|$ motions $(g, z)$ in $O(d)\times \mathbb{F}_q^d$ such that $\phi_{g, z}(\Delta_{k-1}^i)=\Delta_{k-1}^0$. This comes from the fact that for any copy, we always can find a motion with that property and the stabilizer of each simplex is of the size $\sim |O(d-(k-1))|$. Let $R$ be the set of these motions. So we can assume $|R|\sim |O(d-k+1)|\prod_{i=1}^k|A_i|q^{-\binom{k}{2}}$.

As in the case $k=2$, by using Lemma \ref{fact2}, it is enough to show that $A_{k+1}':=\bigcup_{r\in R}\phi_r^{-1}(A_{k+1})$ is of the size at least $\gg q^d$. 

We have, by Theorem \ref{inci-large}, 
\[|A_{k+1}||R|\le I(A_{k+1}\times A_{k+1}', R^{-1})\le \frac{|A_{k+1}||A_{k+1}'||R|}{q^d}+|O(d-1)|^{1/2}q^{d/2}\sqrt{|A_{k+1}||A_{k+1}'||R|}.\]
This implies 
\[|A_{k+1}'|\gg \min \left\lbrace q^d, ~\frac{\prod_{i=1}^{k+1}|A_i||O(d-k+1)|}{q^d|O(d-1)|q^{\binom{k}{2}}}\right\rbrace\gg q^d,\]
whenever $\prod_{i=1}^{k+1}|A_i|\gg q^{dk+1}$.
\end{proof}
\section{Rich rigid motions and simplices (Theorems \ref{thm:positive} and \ref{improve-2d})}\label{sec:5}
In this section, we prove Theorems \ref{thm:positive} and \ref{improve-2d}. 

Let $P=A\times A\subset \mathbb{F}_q^d\times \mathbb{F}_q^d$ and $R$ be a set of rigid motions. For $r=(g, z)\in O(d)\times \mathbb{F}_q^d$, let $i(r)$ be the number of pairs $(x, y)\in P$ such that $gx+z=y$. Let $R_k$ be the set of elements of $R$ incident to at least $k$ elements from $P$. Using upper bounds of $R_k$ given in
Corollaries \ref{co-rich-d} and \ref{co-rich-2}, in the following propositions, we bound the sum $\sum_{r}i(r)^t$, for $t\ge 3$, from above, where the sum runs over all rigid motions in $O(d)\times \mathbb{F}_q^d$.
\begin{proposition}\label{expo-t}
Let $P=A\times A$ with $A\subset \mathbb{F}_q^d$. For $r=(g, z)\in O(d)\times \mathbb{F}_q^d$, let $i(r)$ be the number of pairs $(x, y)\in P$ such that $gx+z=y$. Then, for $t\ge 3$, we have 
    \[\sum_{r}i(r)^t\ll \frac{|P|^t|O(d)|}{q^{(t-1)d}}+q^d|O(d-1)||P|^{t/2}.\]
\end{proposition}
\begin{proof}
Let $R_k$ be the set of elements of $R$ incident to at least $k$ elements from $P$. Assume $k>2|P|q^{-d}$, then, we know from Corollary \ref{co-rich-d} that
\[|R_k|\ll \frac{|O(d-1)||P|q^d}{k^2}.\]

Using the fact that $i(r)\le |A|$ for all $r\in O(d)\times \mathbb{F}_q^d$, we have
    \begin{align*}
        &\sum_{r\in O(d)\times \mathbb{F}_q^d}i(r)^t=\sum_{r, i(r)<2|P|q^{-d}}i(r)^t+\sum_{r, i(r)>2|P|q^{-d}}i(r)^t\\
        &\ll\left(\frac{|P|}{q^d}\right)^{t}\cdot |O(d)|\cdot q^d+q^d|P||O(d-1)||A|^{t-2}\\
        &\ll \frac{|P|^t|O(d)|}{q^{(t-1)d}}+q^d|O(d-1)||P|^{t/2}.
    \end{align*}
Thus, the theorem follows.
\end{proof}
In two dimensions and $q\equiv 3\mod 4$, we have a stronger estimate. 
\begin{proposition}\label{rich-2d}
Let $P=A\times A$ with $A\subset \mathbb{F}_q^2$ and $q\equiv 3\mod 4$. For $r=(g, z)\in O(2)\times \mathbb{F}_q^2$, let $i(r)$ be the number of pairs $(x, y)\in P$ such that $gx+z=y$. Then, for $t\ge 2$, we have
    \[\sum_{r}i(r)^t\ll \frac{|P|^t}{q^{2t-3}}+q|P|^{\frac{2t+1}{4}}.\]
\end{proposition}
\begin{proof}
Let $R_k$ be the set of elements of $R$ incident to at least $k$ elements from $P$. Assume $k>2|P|q^{-2}$, then we know from Corollary \ref{co-rich-2} that
\[|R_k|\ll \frac{q|P|^{5/4}}{k^2}.\]
Using the fact that $i(r)\le |A|$ for all $r\in O(2)\times \mathbb{F}_q^2$, one has
    \begin{align*}
        &\sum_{r\in O(2)\times \mathbb{F}_q^2}i(r)^t=\sum_{r, i(r)<2|P|q^{-2}}i(r)^t+\sum_{r, i(r)>2|P|q^{-2}}i(r)^t\\
        &\le |O(2)|\cdot q^2\cdot \frac{|P|^{t}}{q^{2t}}+q|P|^{5/4}|A|^{t-2}\\
        &\ll \frac{|P|^{t}}{q^{2t-3}}+q|P|^{\frac{5}{4}+\frac{t-2}{2}}=\frac{|P|^t}{q^{2t-3}}+q|P|^{\frac{2t+1}{4}}.
    \end{align*}
This completes the proof.
\end{proof}
We are now ready to prove Theorems \ref{thm:positive} and \ref{improve-2d}.
\begin{proof}[Proof of Theorem \ref{thm:positive}]
We first start with the case $k\le d$. We recall that two $k$-simplices $\mathbf{x}=(x^1, \ldots, x^{k+1})$ and $\mathbf{y}=(y^1, \ldots, y^{k+1})$ are in the same congruence class if there exist $g\in O(d)$ and $z\in \mathbb{F}_q^d$ such that $gx^i+z=y^i$. We denote the set of equivalence classes by $\Delta_k(A)$. For each $C\in \Delta_k(A)$, by $|C|$ we mean the number of $k$-simplices in that class. 

We first observe that 
\[|A|^{k+1}=\sum_{C\in \Delta_k(A)}|C|.\]
Thus, by the Cauchy-Schwarz inequality, we have 
\[|A|^{2(k+1)}\le |\Delta_k(A)|\cdot \#\left\lbrace (\mathbf{x}, \mathbf{y})\in A^{2k+2}\colon \mathbf{x}\sim \mathbf{y} \right\rbrace.\]
In the next step, we need to bound the number of pairs $(\mathbf{x}, \mathbf{y})\in A^{2k+2}$ such that $\mathbf{x}\sim \mathbf{y}$. For each $k$-simplex $\mathbf{x}=(x^1, \ldots, x^{k+1})$, by $\dim(\mathbf{x})$ we mean the dimension of the space spanned by vectors $\{x^2-x^1, \ldots, x^{k+1}-x^1\}$. For each $C\in \Delta_k(A)$, we denote the stabilizer size of simplices in $C$ by $s(C)$. If $\dim(\mathbf{x})=\gamma$ and $\mathbf{x}\in C$, then we know that $s(C)\sim |O(d-\gamma)|$, see \cite{iosevich-Forum} for instance. Thus, for all $C\in \Delta_k(A)$, one has $s(C)\ge |O(d-k)|$.

For each $r=(g, z)\in O(d)\times \mathbb{F}_q^d$, as in Proposition \ref{expo-t}, let $i(r)$ be the number of pairs $(u, v)\in A\times A$ such that $gu+z=v$. It is clear that 
\begin{align*}
    &\#\left\lbrace (\mathbf{x}, \mathbf{y})\in A^{2k+2} \colon \mathbf{x}\sim \mathbf{y}\right\rbrace\lesssim \frac{1}{|O(d-k)|}\sum_{r\in O(d)\times \mathbb{F}_q^d}i(r)^{k+1}\\ &\lesssim \frac{1}{|O(d-k)|}\cdot \left(\frac{|P|^{k+1}|O(d)|}{q^{dk}}+q^d|O(d-1)||P|^{(k+1)/2}\right)\\
    &\lesssim \frac{|A|^{2k+2}|O(d)|}{q^{\binom{d-k}{2}+dk}}+q^{d+\binom{d-1}{2}-\binom{d-k}{2}}|A|^{k+1}\ll \frac{|A|^{2k+2}}{q^{\binom{k+1}{2}}},
\end{align*}
whenever $|A|\gg q^{\frac{dk+1}{k+1}}$.

In other words, under this condition, the number of congruence classes of $k$-simplex in $A$ is at least $\gg q^{\binom{k+1}{2}}$. 

On the other hand, it has been proved in \cite{iosevich-Forum} that, the number of congruence classes of degenerate $k$-simplex in $\mathbb{F}_q^d$ is at most $o\left(q^{\binom{k+1}{2}} \right)$.

We now move to the case $k>d$. In this case, we say a $k$-simplex $\mathbf{x}=(x^1, \ldots, x^{k+1})$ is non-degenerate if $\dim(\mathbf{x})=d$, and degenerate otherwise. It has been proved in \cite{Mc-d} that for $k>d$, the number of congruence classes of $k$-simplex in $\mathbb{F}_q^d$ in total is $\sim q^{d(k+1)-\binom{d+1}{2}}$. It was also proved in \cite{Mc-d} that the number of degenerate $k$-simplex is $o\left(q^{d(k+1)-\binom{d+1}{2}}\right)$.

When $k>d$, we proceed as above, the only difference here is that 
\begin{align*}
    &\#\left\lbrace (\mathbf{x}, \mathbf{y})\in E^{2k+2} \colon \mathbf{x}\sim \mathbf{y}\right\rbrace\lesssim\sum_{r\in O(d)\times \mathbb{F}_q^d}i(r)^{k+1}\\ &\lesssim \frac{|P|^{k+1}|O(d)|}{q^{dk}}+q^d|O(d-1)||P|^{(k+1)/2}\\
    &\lesssim \frac{|A|^{2k+2}|O(d)|}{q^{dk}}+q^{d+\binom{d-1}{2}-\binom{d-k}{2}}|A|^{k+1}\ll \frac{|A|^{2k+2}}{q^{d(k+1)-\binom{d+1}{2}}},
\end{align*}
whenever $|A|\gg q^{\frac{dk+1}{k+1}}$.

This completes the proof of the theorem. 
\end{proof}

\begin{proof}[Proof of Theorem \ref{improve-2d}]
The proof of Theorem \ref{improve-2d} is the same as that of dimension $d$, except that we use Proposition \ref{rich-2d} in place of Proposition \ref{expo-t}. More precisely, we will have
\begin{align*}
    &\#\left\lbrace (\mathbf{x}, \mathbf{y})\in A^{2k+2} \colon \mathbf{x}\sim \mathbf{y}\right\rbrace\lesssim \sum_{r\in O(2)\times \mathbb{F}_q^2}i(r)^{k+1}\\ 
    &\ll \frac{|A|^{2k+2}}{q^{2k-1}}+q|A|^{\frac{2k+3}{2}}\ll \frac{|A|^{2k+2}}{q^{2k-1}},
\end{align*}
whenever $|A|\gg q^{\frac{4k}{2k+1}}$. In conclusion, under $|A|\gg q^{\frac{4k}{2k+1}}$, we have $|\Delta_k(A)|\gg q^{2k-1}$. This completes the proof of the theorem.
\end{proof}

We conclude this section with some discussions on the case of different underlying sets. For the sake of simplicity, let us consider the case $k=2$ in $\mathbb{F}_q^d$. Let $A, B, C$ be sets in $\mathbb{F}_q^d$, to prove that $A\times B\times C$ determines a positive proportion of all congruence classes of triangles, we need to show that  
\begin{align*}
    &\#\left\lbrace (\mathbf{x}, \mathbf{y})\in (A\times B\times C)^2 \colon \mathbf{x}\sim \mathbf{y}\right\rbrace\lesssim \frac{1}{|O(d-2)|}\sum_{r\in O(d)\times \mathbb{F}_q^d}i_A(r)i_B(r)i_C(r)\\ &\lesssim \frac{1}{|O(d-2)|}\cdot \left(\frac{|A|^2|B|^2|C|^2|O(d)|}{q^{2d}}\right),
\end{align*}
here $i_U(r)$ is the number of pairs $(x, y)\in U\times U$ such that $\phi_r(x)=gx+z=y$.

By the H\"{o}lder inequality, we have 
\[\sum_{r\in O(d)\times \mathbb{F}_q^d}i_A(r)i_B(r)i_C(r)\le \left(\sum_ri_A(r)^3\right)^{1/3}\cdot \left(\sum_ri_B(r)^3\right)^{1/3}\cdot \left(\sum_ri_C(r)^3\right)^{1/3}.\]
Applying Proposition \ref{expo-t}, we need the conditions that 
\[|A|, |B|, |C|\gg q^{\frac{2d+1}{3}}.\]
If we use different exponents in the H\"{o}lder step, namely, 
\[\sum_{r\in O(d)\times \mathbb{F}_q^d}i_A(r)i_B(r)i_C(r)\le \left(\sum_ri_A(r)^4\right)^{1/4}\cdot \left(\sum_ri_B(r)^4\right)^{1/4}\cdot \left(\sum_ri_C(r)^2\right)^{1/2},\]
then we will need
\[|A|, |B|\gg q^{\frac{3d+1}{4}}, |C|\gg q^{\frac{d+1}{2}}.\]
In other words, compared to these results, Theorem \ref{moregeneral} offers the best conditions in practice. 
\section{Furstenberg type problem for rigid motions (Theorems \ref{Fur1}, \ref{Fur2}, \ref{Fur3}, and \ref{Fur4})}
The proofs of Theorems \ref{Fur1}, \ref{Fur2}, \ref{Fur3}, and \ref{Fur4} are the same. Thus, we only present that of Theorem \ref{Fur1}. 

More precisely, we set $B=\bigcup_{r\in R}\phi_r(A)$, then it is clear that 
\[I(A\times B, R)=|R||A|.\]
Applying Theorem \ref{inci-large}, one has 
\[I(A\times B, R)\le \frac{|A||B||R|}{q^d}+Cq^{\frac{d}{2}}\sqrt{|O(d-1)|}|A|^{1/2}|B|^{1/2}|R|^{1/2}.\]
Putting lower and upper bounds together, we get the desired result. 

We note here that in the proofs of Theorems \ref{Fur2}-\ref{Fur4}, we use Theorems \ref{thm: incidence1}, \ref{thm-inci-large-2}, and \ref{thm-inci}, respectively. 
\section{Incidences for small sets (Theorem \ref{thm-inci})}

Let $p_1, p_2, p_3, p_4\in \mathbb{F}_q^2$ be points such that $||p_1-p_3||=||p_2-p_4||\ne 0$ and $p_1-p_3\ne p_2-p_4$. We know that there exists a unique pair $(g, z)$ with $g\in SO(2, q)\setminus\{I\}$ and $z\in \mathbb{F}_q^2$ such that $gp_1+z=p_2$ and $gp_3+z=p_4$.

Since $q\equiv 3\mod 4$, it has been indicated in \cite{alex} that there is a natural way to identify the group $SO(2, q)$ and the field $\mathbb{F}_q$. More precisely, we consider the map $\varphi\colon \mathbb{F}_q\to SO(2, q)\setminus \{I\}$ defined by 
\[\varphi(r)=\begin{pmatrix}
    \frac{r^2-1}{r^2+1}&\frac{-2r}{r^2+1}\\
    \frac{2r}{r^2+1}&\frac{r^2-1}{r^2+1}
\end{pmatrix}.\]
Note that $r^2+1\ne 0$ under $q\equiv 3\mod 4$. From now we identify each matrix in $SO(2, q)\setminus\{I\}$ with the corresponding element in $\mathbb{F}_q^2$. 

Given $p, q\in \mathbb{R}^2$, it is well-known that the set of rigid motions in $\mathbb{R}^2$ that map $p$ to $q$ can be parameterized as a line in $\mathbb{R}^3$, see \cite{ES, GK} or \cite[Chapter 9]{Gbook} for instance. Thus, the number of incidences between points and rigid motions in $\mathbb{R}^2$ can be reduced to the number of incidences between points and lines in $\mathbb{R}^3$. To prove Theorem \ref{thm-inci}, we will apply the same strategy. We first recall the following lemma from \cite{alex}.
\begin{lemma}[\cite{alex}]\label{8.2}
Given $p, q\in \mathbb{F}_q^2$. Define $\ell_{p\to q}$ be the set of points $\left((I-g)^{-1}z, r\right)$, where $gp+z=q$, $g=\varphi(r)\in SO(2, q)\setminus \{I\}$, $z\in \mathbb{F}_q^2$, then $\ell_{p\to q}$ is a line in $\mathbb{F}_q^3$. In particular, $\ell_{p\to q}$ can be presented in the following form 
\[\left\lbrace \left(\frac{p+q}{1}, 0\right)+r\left(\frac{(p-q)^\perp}{2}, 1\right)\colon r\in \mathbb{F}_q \right\rbrace.\]
Here $(p_1, p_2)^\perp=(p_2, -p_1)$.
\end{lemma}
In $\mathbb{F}_q^3$, the following point-line incidence bound is due to Koll\'{a}r in \cite{kollar}.
\begin{theorem}[\cite{kollar}]\label{8.3}
    Let $P$ be a set of $n$ distinct points and $L$ be a set of $m$ distinct lines in $\mathbb{F}_q^3$. Assume that no plane contains more than $c\sqrt{m}$ lines from $L$ for some constant $c$. Then we have 
    \[I(P, L)\ll |L||P|^{2/5}+|P|^{6/5}.\]
\end{theorem}
Using Lemma \ref{8.2} above, one can see that the number of incidences between $P$ and $R$ can be reduced to the number of incidences between points and lines in $\mathbb{F}_q^3$. Let $P$ and $L$ be the set of corresponding points and lines in $\mathbb{F}_q^3$.
So we can apply Theorem \ref{8.3} as long as we know that any plane contains only few lines. The rest of this section is devoted to proving such a property.

We first observe that the number of lines is equal to $|P|=|U|^2$. This means that we need to show that each plane contains at most $\ll |U|$ lines. 

Given a plane $\pi$, it is clear that any family of parallel lines contains at most $|U|$ lines, and the same happens for any pencil, i.e. families of concurrent lines. 

\begin{lemma}
If there exist three lines on $\pi\cap L$ that are non-concurrent and pairwise intersecting, then $\pi$ contains at most $|U|$ lines from $L$.
\end{lemma}
\begin{proof}
Assume those three lines are $\ell_{p_i\to q_i}$, $1\le i\le 3$, then we can check that the two triangles $(p_1, p_2, p_3)$ and $(q_1, q_2, q_3)$ are in the same congruence class under a rigid motion with an orthogonal matrix of determinant $-1$. 

Let 
\[
g=\begin{pmatrix}
&a&b\\
&b&-a
\end{pmatrix}, ~a^2+b^2=1,\]
be an orthogonal matrix with $\det(g)=-1$, and $z:=(z_1, z_2)\in \mathbb{F}_q^2$. For $x=(x_1, x_2)\in \mathbb{F}_q^2$, the map $\phi_{g, z}$ is defined by
\[\phi_{g, z}(x_1, x_2)=
\begin{pmatrix}
a&b\\
b&-a
\end{pmatrix}\cdot \begin{pmatrix}
    x_1\\x_2
\end{pmatrix}+\begin{pmatrix}
    z_1\\z_2
\end{pmatrix}=\begin{pmatrix}
    ax_1+bx_2+z_1\\bx_1-ax_2+z_2
\end{pmatrix}.\]
Under this map,we have 
\[(x_1, x_2)\to (ax_1+bx_2+z_1, bx_1-ax_2+z_2), ~(y_1, y_2)\to (ay_1+by_2+z_1, by_1-ay_2+z_2).\]

By abuse of notation, let $\ell_{(x_1,x_2), \phi_{g, z}}$ be the line defined as in Lemma \ref{8.2} by two points $(x_1, x_2)$ and $\phi_{g, z}(x_1,x_2)$. Then the line $\ell_{(x_1, x_2), \phi_{g, z}}$ contains points of the form 
\[\left(\frac{x_1(a+1)+bx_2+z_1}{2}, \frac{bx_1+x_2(1-a)+z_2}{2}, 0\right)+r\left(\frac{x_2(1+a)-bx_1-z_2}{2}, \frac{(a-1)x_1+bx_2+z_1}{2}, 1\right),\]
and the line $\ell_{(y_1, y_2), \phi_{g, z}}$ contains points of the form
\[\left(\frac{y_1(a+1)+by_2+z_1}{2}, \frac{by_1+y_2(1-a)+z_2}{2}, 0\right)+r\left(\frac{y_2(1+a)-by_1-z_2}{2}, \frac{(a-1)y_1+by_2+z_1}{2}, 1\right),\]
where $r\in \mathbb{F}_q$. Assume these two lines span a plane denoted by $\pi$. 

We now show that given $(u_1, u_2), (v_1, v_2)\in \mathbb{F}_q^2$, if the line containing points of the form 
\[\left(\frac{u_1+v_1}{2}, \frac{u_2+v_2}{2}, 0 \right)+r\left(\frac{u_2-v_2}{2}, \frac{v_1-u_1}{2}, 1\right).\]
lies on $\pi$, then we will have
\[\begin{pmatrix}
    a&b\\b&-a
\end{pmatrix}\cdot \begin{pmatrix}
    u_1\\u_2
\end{pmatrix}+\begin{pmatrix}
    z_1\\z_2
\end{pmatrix}=\begin{pmatrix}
    v_1\\v_2
\end{pmatrix}.\]
This implies that the plane $\pi$ contains at most $|U|$ lines as desired and we are done. 

Since any plane is of dimension two, there exist $(m, n)\ne (0, 0)$ and $m+n=1$ such that 
\[u_2-v_2=m\left((1+a)x_2-bx_1-z_2 \right)+n\left((1+a)y_2-by_1-z_2\right),\]
and 
\[v_1-u_1=m((a-1)x_1+bx_2+z_1)+n((a-1)y_1+by_2+z_1).\]
We then can write 
\begin{align*}
    v_1&=u_1+(a-1)(mx_1+ny_1)+b(mx_2+ny_2)+z_1\\
    &=au_1+bu_2+(a-1)(mx_1+ny_1-u_1)+b(mx_2+ny_2-u_2)+z_1,
\end{align*}
and 
\begin{align*}
    v_2&=u_2-m((1+a)x_2-bx_1-z_2)-n((1+a)y_2-by_1-z_2)\\
    &=bu_1-au_2-(1+a)(mx_2+ny_2-u_2)+b(mx_1+ny_1-u_1)+z_2.
\end{align*}
It suffices to prove that
\[(a-1)(mx_1+ny_1-u_1)+b(mx_2+ny_2-u_2)=0, ~(1+a)(mx_2+ny_2-u_2)+b(mx_1+ny_1-u_1)=0.\]
Note that the plane $\pi$ intersects the plane $Z=0$ in $\mathbb{F}_q^3$ at the line denoted by $\ell({Z=0})$ defined by 
\[\left\lbrace\left(\frac{(a+1)x+by+z_1}{2}, \frac{bx+y(1-a)+z_2}{2} \right)\colon (x, y)\in \mathbb{F}_q^2\right\rbrace.\]
To simplify the notations, we assume that $(z_1, z_2)=(0, 0)$. 

Note that the point $((u_1+v_1)/2, (u_2+v_2)/2)$ lies on this line. This point can be rewritten as 
\begin{align*}
&=:\left(\frac{u_1(a+1)+bu_2}{2}, \frac{bu_1-(a-1)u_2}{2}\right)+(\mathcal{A}/2, \mathcal{B}/2),
\end{align*}
where
\[(\mathcal{A}, \mathcal{B}):=((a-1)(mx_1+ny_1-u_1)+b(mx_2+ny_2-u_2), -(a+1)(mx_2+ny_2-u_2)+b(mx_1+ny_1-u_1)).\]
Note that $\left(\frac{u_1(a+1)+bu_2}{2}, \frac{bu_1-(a-1)u_2}{2}\right)\in \ell({Z=0})$ by a direct computation. If \[(u_1(a+1)+bu_2, bu_1-(a-1)u_2)+(\mathcal{A}, \mathcal{B})\in \ell_{Z=0},\] then the direction of $\ell(Z=0)$ is parallel to $(\mathcal{A}, \mathcal{B})$ in the plane $Z=0$.

On the other hand, $(0, 0)\in \ell({Z=0})$ and 
\[\frac{1}{2}\cdot\begin{pmatrix}
    (a+1)&b\\
    b&(1-a)
\end{pmatrix}\cdot \begin{pmatrix}
    -mx_2-ny_2+u_2\\
    mx_1+ny_1-u_1
\end{pmatrix}=(\mathcal{B}/2, -\mathcal{A}/2)\in \ell({Z=0}).\]
This means that $(\mathcal{A}, \mathcal{B})=\lambda (\mathcal{B}, -\mathcal{A})$ for some $\lambda\ne 0$. This infers $\mathcal{A}=\mathcal{B}=0$. 

In other words, any line of the form 
\[\left(\frac{u_1+v_1}{2}, \frac{u_2+v_2}{2}, 0 \right)+r\left(\frac{u_2-v_2}{2}, \frac{v_1-u_1}{2}, 1\right)\]
belonging to $\pi$ will satisfy the relation
\[\begin{pmatrix}
    a&b\\b&-a
\end{pmatrix}\cdot \begin{pmatrix}
    u_1\\u_2
\end{pmatrix}+\begin{pmatrix}
    z_1\\z_2
\end{pmatrix}=\begin{pmatrix}
    v_1\\v_2
\end{pmatrix}.\]
This completes the proof.
\end{proof} 
To conclude that the intersection $\pi\cap L$ contains at most $\ll |U|$ lines, we decompose $\pi\cap L$ into families of parallel lines, say, $F_1, F_2, \ldots, F_m$. If $m\le 2$, we are done. Thus, we can assume that $m\ge 3$. If there are three lines satisfying the above lemma, we are also done. Thus, we can assume that $\pi\cap L$ contains no three such lines. Let $\ell_1\in F_1$ and $\ell_2\in F_2$. Set $p=\ell_1\cap \ell_2$, then all lines in $F_i$, $i\ge 3$, have to pass through $p$. This means that $|F_i|\le 1$ for all $i\ge 3$. Using the fact that for each $p$, there are at most $|U|$ lines containing it, so in total, the intersection $\pi\cap L$ contains at most $\ll |U|$ lines.
\section{Discussions over prime fields}
We might wonder if the recent progress on the distance problem over prime fields due to Murphy, Petridis, Pham, Rudnev, and Stevens \cite{MPPRS} is helpful to study the simplex problem. In this section, we discuss this direction. We adapt the method in Section \ref{sec:5} to prove the following theorem. 
\begin{theorem}\label{1.13}
    Let $A, B, C\subset \mathbb{F}_p^2$ with $p^{\frac{5}{4}}\le |C|\le p^{\frac{4}{3}}$ and $p\equiv 3\mod 4$. Assume that $|A|, |B|\gg p^{12/7}$, then the number of congruence classes of triangles with vertices in $A\times B\times C$ is at least $\gg p^3$.  
\end{theorem}
\begin{remark}
    Compared to Theorem \ref{firstmaintheorem}, this result is much weaker. In Theorem \ref{1.13}, we only assumed that $|C|\in (p^{5/4}, p^{4/3})$, if the same condition applies for $A$ or $B$, then the condition will become worse. When $A=B=C$ (for simplicity) and $|A|=|B|=|C|<p^{5/4}$, it is not difficult to see that there will be at least $|\Delta(A)|\cdot |A|$ congruence classes with vertices in $A\times B\times C$. Here $\Delta(A)$ is the set of distinct distances determined by pairs of points in $A$. 
\end{remark}
We now prove Theorem \ref{1.13}.

In the plane over prime fields, the following incidence theorem was proved in \cite{PY}. 
\begin{theorem}[\cite{PY}]\label{thm-inci-large-3}
Let $P=U\times U\subset \mathbb{F}_p^2\times \mathbb{F}_p^2$ and $R$ be a subset of $O(2)\times \mathbb{F}_p^2$ with $p\equiv 3\mod 4$. Assume that $p^{5/4}\le |U|\le p^{4/3}$, then we have
\[I(P, R)\le \frac{|P||R|}{p^2}+Cp^{1/8}|U|^{3/2}|R|^{1/2},\]
for some large positive constant $C$.
\end{theorem}
With the same argument as in arbitrary finite fields, we have the following corollary. 
\begin{corollary}\label{co-rich-222}
Let $P=U\times U\subset \mathbb{F}_p^2\times \mathbb{F}_p^2$ with $p\equiv 3\mod 4$ and $p^{5/4}\le |U|\le p^{4/3}$. For $k>2|P|p^{-2}$, let $R_k$ be the number of elements of $R$ incident to at least $k$ elements from $P$. Then we have 
\[|R_k|\ll \frac{p^{1/4}|U|^3}{k^2}.\]
For $r=(g, z)\in O(2)\times \mathbb{F}_p^2$, let $i_U(r)$ be the number of pairs $(x, y)\in U\times U$ such that $gx+z=y$. Then, for $t\ge 2$, we have
    \[\sum_{r}i(r)^t\ll \frac{|P|^t}{p^{2t-3}}+p^{1/4}|U|^{t+1}.\]
\end{corollary}
\begin{proof}[Proof of Theorem \ref{1.13}]
As in the case of $\mathbb{F}_q$, we have 
\begin{align*}
    &\#\left\lbrace (\mathbf{x}, \mathbf{y})\in (A\times B\times C)^2 \colon \mathbf{x}\sim \mathbf{y}\right\rbrace\lesssim \sum_{r\in O(2)\times \mathbb{F}_q^2}i_A(r)i_B(r)i_C(r)\\ 
    &\le \left(\sum_{r}i_A(r)^4\right)^{1/4}\cdot  \left(\sum_{r}i_B(r)^4\right)^{1/4}\cdot  \left(\sum_{r}i_C(r)^2\right)^{1/2}.\\
\end{align*}
It is sufficient to show that 
\[\#\left\lbrace (\mathbf{x}, \mathbf{y})\in (A\times B\times C)^2 \colon \mathbf{x}\sim \mathbf{y}\right\rbrace\ll \frac{|A|^2|B|^2|C|^2}{p^3}\]
when $|A|, |B|\gg p^{12/7}$. 

Using Proposition \ref{rich-2d}, we have 
\[\sum_{r}i_A(r)^4\ll \frac{|A|^8}{p^5}+p|A|^{9/2},\]
and 
\[\sum_{r}i_B(r)^4\ll \frac{|B|^8}{p^5}+p|B|^{9/2}.\]
The above corollary gives 
\[\sum_{r}i_C(r)^2\ll \frac{|C|^4}{p}.\]
In total, 
\[\#\left\lbrace (\mathbf{x}, \mathbf{y})\in (A\times B\times C)^2 \colon \mathbf{x}\sim \mathbf{y}\right\rbrace\ll \frac{|A|^2|B|^2|C|^2}{p^3}+|A|^{9/8}|B|^{9/8}|C|^2+\frac{|A|^2|B|^{9/8}|C|^2}{p^{3/2}}+\frac{|B|^2|A|^{9/8}|C|^2}{p^{3/2}}.\]
This is at most $\ll |A|^2|B|^2|C|^2/p^3$ whenever $|A|, |B|\gg p^{12/7}$. This completes the proof of the theorem.

We want to add a remark here that if we use the bound 
\[ \sum_{r\in O(2)\times \mathbb{F}_q^2}i_A(r)i_B(r)i_C(r)\le \left(\sum_ri_A(r)^3\right)^{1/3}\cdot \left(\sum_ri_B(r)^3\right)^{1/3}\cdot \left(\sum_ri_C(r)^3\right)^{1/3}\]
 then we will never get a positive proportion of all congruence classes of triangles no matter how large $A$ and $B$ are. This happens when the size of $C$ is smaller than $p^{3/2}$.
\end{proof}

\section{Acknowledgements}
T. Pham would like to thank the Vietnam Institute for Advanced Study in Mathematics (VIASM) for the hospitality and for the excellent working condition. 
\bibliographystyle{amsplain}

\end{document}